\documentclass[12pt,a4paper,psamsfonts]{amsart}
\usepackage{fancyhdr}
\usepackage{appendix}
\usepackage{amssymb,amscd,amsxtra,calc}
\usepackage{mathrsfs}
\usepackage{cmmib57}
\usepackage{multirow}
\usepackage[all]{xy}
\usepackage{longtable}
\usepackage[colorlinks,linkcolor=blue,anchorcolor=blue,citecolor=blue]{hyperref}
\usepackage{tikz}

\theoremstyle{plain}
    \newtheorem{thm}{Theorem}[section]

     \newtheorem{conjecture}[thm]{Conjecture}
    \newtheorem{corollary}[thm]{Corollary}
    
    \newtheorem{lemma}[thm]{Lemma}
    \newtheorem{proposition}[thm]{Proposition}
    \newtheorem{question}[thm]{Question}
    \newtheorem{theorem}[thm]{Theorem}

\theoremstyle{definition}
    \newtheorem{definition}[thm]{Definition}

    \newtheorem{notation}[thm]{Notation}
    \newtheorem*{notation*}{Notation and Terminology}
      
    \newtheorem{remark}[thm]{Remark}
    
\theoremstyle{remark}

\newcommand{\bP}{\mathbb{P}}
\newcommand{\bQ}{\mathbb{Q}}
\newcommand{\bR}{\mathbb{R}}

\newcommand{\bZ}{\mathbb{Z}}

\newcommand{\alb}{\operatorname{alb}}
\newcommand{\Amp}{\operatorname{Amp}}

\newcommand{\Aut}{\operatorname{Aut}}

\newcommand{\Bir}{\operatorname{Bir}}

\newcommand{\id}{\operatorname{id}}

\newcommand{\Nef}{\operatorname{Nef}}

\newcommand{\NS}{\operatorname{NS}}

\newcommand{\N}{\operatorname{N}}

\newcommand{\Alb}{\operatorname{Alb}}

\newcommand{\Pic}{\operatorname{Pic}}

\newcommand{\arxiv}[1]{\href{https://arxiv.org/abs/#1}{{\tt arXiv:#1}}}
\newcommand{\mstriangle}[1]{
\begin{tikzpicture}[x=0.3cm,y=0.3cm]
\draw (-0.4,-0.433) -- (1.4,-0.433);
\draw (-0.2,-0.7794) -- (0.7,0.7794);
\draw (1.2,-0.7794) -- (0.3,0.7794);
\end{tikzpicture}
}
\newcommand{\mssharp}[1]{
\begin{tikzpicture}[x=0.3cm,y=0.3cm]
\draw (-0.8,-0.5) -- (0.8,-0.5);
\draw (-0.8,0.5) -- (0.8,0.5);
\draw (-0.5,-0.8) -- (-0.5,0.8);
\draw (0.5,-0.8) -- (0.5,0.8);
\end{tikzpicture}
}

\makeatletter

\newcommand{\Rmnum}[1]{\expandafter\@slowromancap\romannumeral #1@}
\makeatother

\begin{document}

\title[Kawaguchi-Silverman conjecture]
{Kawaguchi-Silverman conjecture on automorphisms of projective threefolds}

\author{Sichen Li}
\address{
School of Mathematics, East China University of Science and Technology, Shanghai 200237, P. R. China}
\address{Shanghai Key Laboratory of Pure Mathematics and Mathematical Practice, East China Normal University, Shanghai, 200241, P. R. China}
\email{\href{mailto:sichenli@ecust.edu.cn}{sichenli@ecust.edu.cn}}
\begin{abstract}
Under the framework of dynamics on normal projective varieties by Kawamata, Nakayama and Zhang \cite{Kawamata85,Nakayama10,NZ09,NZ10,Zhang16}, Hu and the author \cite{HL21}, we may reduce Kawaguchi-Silverman conjecture for automorphisms $f$ on normal projective threefolds $X$  with  either the canonical divisor $K_X$ is trivial or negative Kodaira dimension to the following two case: (i)  $f$ is a primitively automorphism of a weak Calabi-Yau threefold (ii) $X$ is a rationally connected threefold.
And we prove Kawaguchi-Silverman conjecture is true for automorphisms of normal projective varieties $X$ with the irregularity $q(X)\ge\dim X-1$.
Finally, we discuss Kawaguchi-Silverman conjecture on normal projective varieties with Picard number two.
\end{abstract}

\dedicatory{Dedicated to Professor De-Qi Zhang on the occasion of his 60th birthday}
\thanks{The research is supported by  Shanghai Sailing Program (No. 23YF1409300) and Science and Technology Commission of Shanghai Municipality (No. 22dz2229014).}
 \maketitle 
 \section{Introduction}
\subsection{Kawaguchi-Silverman conjecture}
Let $X$ be a normal projective variety of dimension $n\ge1$  over a  global field $K$ of characteristic 0.
Let $f: X\dashrightarrow X$ be a dominant rational map.
Then the {\it $k$-th dynamical degree of  $f$} (cf. \cite{Dang20,DS05,Truong20}) is defined as
\begin{equation*}
                                   d_k(f):=\lim_{s\to\infty}\big((f^s)^*(H^k)\cdot H^{n-k}\big)^{1/s},
\end{equation*}
where $H$ is an ample Cartier divisor on $X$.
A result of Dinh and Sibony \cite{DS05} says that this limit exists and is independent of the choice of  the ample divisor $H$.

In \cite{KS16-tams,KS16}, Kawaguchi and Silverman studied another dynamical invariant called the {\it arithmetic degree}, which we now recall.
Assume that $X$ and $f$ are defined over the algebraic closure $\overline{\mathbb Q}$ of the rational number field $\bQ$, and write $X(\overline{\mathbb Q})_f$ for the set of points $P$ whose forward $f$-orbit
\begin{equation*}
                                    \mathcal O_f(P)=\big\{ P, f(P), f^2(P),\cdots\big\}
\end{equation*}
is well-defined.
Further, let
\begin{equation*}
                        h_H: X(\overline{\mathbb Q})\to [0,\infty)
\end{equation*}
be a Weil height on $X$ associated with an ample divisor $H$ (cf. \cite{BG06,HS00,Lang83}) , and let $h_H^+=\max\big\{1, h_H\big\}$.
For $P\in X(\overline{\mathbb Q})_f$, we define {\it the lower and the upper arithmetic degree of $f$ at $x$} by
$$
\underline{\alpha}_f(P)=\liminf_{n\to\infty}h_H^+(f^n(P))^{\frac{1}{n}},
$$ 
$$
\overline{\alpha}_f(P)=\limsup_{n\to\infty}h_H^+(f^n(P))^{\frac{1}{n}}.
$$
Both of there quantities are independent of the choice of ample divisor $H$ and height function $h_H$ (\cite[Proposition 12]{KS16}).
When $\overline{\alpha}_f(P)=\underline{\alpha}_f(P)$, we write
$$
           \alpha_f(P)=\lim_{n\to+\infty}h_H^+(f^n(P))^{\frac{1}{n}}
$$
and call it the arithmetic degree.
When $f$ is a morphism, the above limit exists by \cite{KS16}.

The following Kawaguchi-Silverman conjecture (KSC) asserts that for a dominant self-map $f: X\dashrightarrow X$ of a projective variety $X$ over $\overline{\mathbb Q}$, the arithmetic degree $\alpha_f(x)$ of any point $x$ with Zariski dense $f$-orbit is equal to the first dynamical degree $d_1(f)$ of $f$.
\begin{conjecture}
\cite[Conjecture 6]{KS16}
\label{KSC}
Let $f: X\dashrightarrow X$ be a dominant rational map of a normal projective variety $X$ over $\overline{\mathbb Q}$, and let $x\in X(\overline{\mathbb Q})_f$.
If $\mathcal O_f(x)$ is Zariski dense in $X$ and there exists the arithmetic degree $\alpha_f(x)$ at $x$, then $\alpha_f(x)=d_1(f)$.
\end{conjecture}
\begin{remark}\label{non-dense}
It is known that any dominant rational self-map of a normal projective variety $X$ with the Kodaira dimension $\kappa(X)>0$  does not have any Zariski dense forward orbits (cf. \cite[Theorem A]{NZ09} and \cite[Theorem 14.10]{Ueno75}).
It is known that $1\le \underline{\alpha}_f(x)\le \overline{\alpha}_f(x)\le d_1(f)$ by \cite[Theorem 4]{KS16} and \cite[Theorem 1.4]{Matsuzawa20b}.
Then KSC is true for $f$ if $d_1(f)=1$.
So,  we may assume that $d_1(f)>1$.
\end{remark}
\subsection{Historical note}\label{History}
Let $f$ be a surjective endomorphism of a  normal projective variety $X$.
When $f$ is an automorphism, one can further take an $f$-equivariant resolution (cf. \cite[Theorem 13.2]{BM97}).
Kawaguchi and Silverman  showed in \cite[Theorem 2(c)]{KS14} that KSC holds for such $f$ when $\dim X=2$.
When $f$ is non-isomorphic and $X$ is smooth, Matsuzawa, Sano and Shibata \cite[Theorem 1.3]{MSS18} proved that KSC holds for $f$, by reducing the problem to three precise cases: $\mathbb P^1$-bundles, hyperelliptic surfaces, and surfaces of Kodaira dimension one.
For a singular projective surface $X$, running an $f$-equivariant minimal model program (MMP) after iterating $f$,  Meng and Zhang \cite[Theorem 1.3]{MZ19} proved that KSC holds for any surjective endomorphism of a projective surface.

Assume that $f$ is a surjective endomorphism of a normal projective variety $X$ and $d_1(f)>1$.
Kawaguchi and Silverman \cite[Theorem 5]{KS16} proved KSC when $f$ is polarized, i.e. there is an ample Cartier divisor $D$ and integer $q>1$ such that $f^*D\sim qD$.
Matsuzawa, Meng, Shibata, Zhang and Zhong \cite[Theorem 1.9(1)]{MMSZZ22} proved that there exists a nef Cartier $\bR$-divisor $D'$ such that $f^*D'\sim d_1(f)D'$.
And Matsuzawa and the author \cite[Proposition 5.2(1)]{LM21} proved that KSC holds for $f$ if the Iitaka dimension $\kappa(X,D)>0$.

Silverman \cite[Theorem 1.2]{Silverman17} proved that KSC holds for any dominant self-map on abelian varieties.
Also,  KSC holds for any  self-morphism  on semi-abelian varieties by Matsuzawa and Sano \cite[Theorem 1.1]{MS20}.

Matsuzawa \cite{Matsuzawa20a} proved that KSC is true when $X$ is a projective toric variety, a linear algebraic group or a variety of Fano type.
Moreover, he  established in \cite[Theorem 4.1]{Matsuzawa20a} that KSC is true for any surjective endomorphism $f$ on $X$ when $\NS_{\bQ}(X)\cong\Pic_{\bQ}(X)$ and the nef cone is generated by finitely many semi-ample integral divisors.

Assume that $f$ is a surjective endomorphism of a projective variety $X$.
In \cite[Theorem 1.2 and Proposition 1.7]{LS21}, Lesieutre and Satriano proved that KSC is true when $X$ is  a  hyper-K\"ahler variety, a smooth projective threefold with $\kappa(X)=0$ and $\deg f>1$.
Moreover, KSC is true when $f$ is an automorphism of a smooth projective variety $X$ with Picard number $\rho(X)=2$ by Shibata \cite[Theorem 4.2]{Shibata19} or \cite[Theorem 2.30]{LS21}.

 When $X$ is a smooth rationally connected variety admitting an int-amplified endomorphism, KSC holds for every surjective endomorphism of such $X$ by Meng and Zhang \cite[Theorem 1.11]{MZ19} and Matsuzawa and Yoshikawa \cite[Theorem 1.1]{MY22}.
Moreover, let $X$ be a $\mathbb Q$-factorial klt projective variety with the algebraic fundamental group $\pi_1^{\mathrm{alg}}(X_{\mathrm{reg}})$ of the smooth locus $X_{\mathrm{reg}}$ of $X$ being finite.
Suppose $X$ admits an int-amplified endomorphism.
Meng, Matsuzawa, Shibata and Zhang \cite[Theorem 6.3]{MMSZ23} proved that KSC holds for any surjective endomorphism of such $X$ by \cite[Theorem 1.7]{MZ19}.

Let $\pi:X\to Y$ be a surjective endomorphism of  normal projective variety.
Suppose $f$ (resp. $g$) is a surjective endomorphisms  of $X$ (resp. $Y$) such that $\pi \circ  f=g \circ \pi$.
It is well known that KSC holds for $f$ if (1) $d_1(f)=d_1(g)$ and (2) KSC holds for $g$ (cf. \cite[Theorem 3.4]{LS21}).
Now let $X$ be a projective bundle over a smooth projective variety $Y$ with the Picard number one.
Lesieutre and Satriano \cite{LS21} proved that KSC is true for any surjective endomorphism of $X$ when $Y=\bP^1$.
Also, Matsuzawa and the author  \cite{LM21} proved that KSC is true for any surjective endomorphism of $X$ when $Y$ is Fano.

Chen, Lin and Oguiso \cite{CLO22} showed that KSC is true for three cases: (i)   birational automorphisms of  a smooth projective variety $X$ with $\kappa(X)=0$ and the irregularity $q(X)\ge \dim X-1$ (ii) birational automorphisms of  an irregular smooth threefold $X$ modulo the case that $X$ is covered by rational surfaces (iii) automorphisms of an irregular smooth threefold $X$.

A dominant rational self-map on a projective variety is called $p$-cohomological hyperbolic if the $p$-th dynamical degree is strictly larger than other dynamical degrees.
Matsuzawa and Wang \cite{MW22} showed that KSC is true for a 1-cohomological hyperbolic of a smooth projective variety, and the arithmetic degrees can be transcendental for dominant rational self-maps by using the striking result of Bell, Diller and Jonsson \cite{BDJ20}.
\subsection{Main results}
Let $f$ be a surjective endomorphism of a normal projective variety $X$.
We say that a rational map $\pi: X\dashrightarrow Y$ is {\it  $f$-equivariant}, if there exists a surjective endomorphism $g$ of $Y$ such that  $\pi\circ f=g\circ\pi$.
When $f$ is an automorphism, there is a resolution of singularities $\pi: X'\to X$ and an automorphism $f'$ on $X'$ such that $\pi\circ f'=f\circ \pi$ (cf. \cite[Theorem 13.2]{BM97}).
Therefore, we reduce KSC for $f$  to KSC for $f'$, which is an automorphism on a smooth projective variety.
However, there is no equivarent minimal model program for automorphism groups of projective varieties in general (cf. \cite[Remark 1.3(1)]{HL21}).
So we consider the special but important case where  $X$ is minimal, e.g. the canonical divisor $K_X\sim 0$ when $\dim X=3$ and $\kappa(X)=0$.

Due to work of H\"oring and Peternell \cite[Theorem 1.5]{HP19}, we have
the Beauville-Bogomolov decomposition for minimal models with trivial canonical class as follows.
Let $X$ be a normal projective variety at most klt singularities such that $K_X\equiv0$.
Then there exists a finite cover, \'etale in codimension one $\pi: \widetilde{X}\to X$  such that
$$
   \widetilde{X}\cong A\times \prod  Y_j \times \prod  Z_k
$$
where $A$ is an abelian variety, the $Y_j$ are singular Calabi-Yau varieties and the $Z_k$ are singular irreducible holomorphic symplectic varieties (see \cite[Definition 1.3]{GGK19}).
However, it is still unclear  whether we can always lift the automorphisms of $X$ to some splitting cover $\widetilde{X}$ (cf. \cite[Remark 3.5]{HL21}). 
Instead of utilizing their strong decomposition theorem,  we use a weak version (cf. \cite[Lemma 2.7]{HL21}) due to Kawamata \cite{Kawamata85}, and developed by Nakayama-Zhang \cite{NZ10}.
For more details about automorphisms of projective varieties with trivial canonical divisor, we refer to \cite[Theorem 1.1 and 2.4]{Zhang16} and \cite[Theorem 1.2]{HL21}.
For automorphisms on projective varieties with negative Kodaira dimension, we  can use special MRC fibration due to Nakayama \cite{Nakayama10} which have the descent property.
We refer  to \cite[Lemma 2.11]{HL21} for more details about it.

The notion of a primitively birational  self-map was introduced by  Zhang \cite{Zhang09-JDG} as follows.
\begin{definition}\label{primitive}
\cite{Zhang09-JDG}
A birational self-map $f: X\to X$ is {\it imprimitive} if there exist a variety $B$ with $1\le\dim B<\dim X$, a birational map $g: B\dashrightarrow B$, and a dominant rational map $\pi: X\dashrightarrow B$ such that $\pi\circ f=g\circ\pi$.
The map $f$ is called {\it primitive} if it is not imprimitive.
\end{definition}
Below is our main result of KSC on automorphisms of normal projective threefolds.
\begin{theorem}\label{reduced-auto}
Let $f$ be an automorphism of a  normal projective threefold with only klt singularities.
Suppose $K_X\sim_\bQ0$ or $\kappa(X)=-\infty$.
Then we reduce  Conjecture \ref{KSC} for $(X,f)$ to the following cases:
\begin{enumerate}
	\item $f$ is a primitively automorphism of a  weak Calabi-Yau threefold.
	\item $X$ is a rationally connected threefold.
\end{enumerate}
\end{theorem}
Notice that  a birational self-map $f$ on a minimal Calabi-Yau threefold $X$ of Picard number $\rho(X)\ge2$ is primitive if the action $f^*|_{\NS_{\bQ}(X)}$ is  irreducible over $\bQ$ (cf. \cite[Corollary 1.3]{Oguiso18}).
This motivates us to ask the following question.
\begin{question}\label{Que}
Let $f$ be a birational self-map of a weak Calabi-Yau variety $X$ with $\rho(X)\ge2$.
Suppose that $f^*|_{\NS_\bQ(X)}$ is irreducible over $\bQ$.
Then is KSC true for $(X,f)$?
\end{question}
Motivated by \cite[Theorems 1.4]{CLO22} and without assuming that the Kodaira dimension vanishes,  we may establish that KSC holds for automorphisms of normal projective varieties  $X$ with $q(X)\ge\dim X-1$ as follows.
\begin{theorem}
\label{q(X)>X-1}	
KSC is true for automorphisms of normal projective varieties $X$ with $q(X)\ge\dim X-1$.
\end{theorem}
\begin{remark}
Let $\pi: X\dasharrow Y$ be dominant rational map of projective varieties.
Let $f: X\dasharrow X$ and $g: Y\dasharrow Y$ be dominant rational self-maps such that $f\circ\pi=\pi\circ g$.
Suppose that $\pi$ is generically finite.
If one could establish that KSC holds for $f$ if and only if KSC holds for $g$, then KSC is true  for any birational automorphism of a normal projective variety $X$ with $q(X)\ge\dim X-1$ by using the same argument in the proof of Theorem \ref{q(X)>X-1}.
\end{remark} 
The paper is organized as follows.
In Section \ref{Preliminaries}, we  collect some basic facts on KSC.
In Section \ref{invariant}, we  study invariant fibrations on projective varieties and prove Theorems  \ref{q(X)>X-1} and \ref{KSC-P1}.
In Section  \ref{MainSect}, we prove Theorems \ref{reduced-auto}.
Finally, we  study  KSC on projective varieties with Picard number two in Section \ref{Pic=2}.

{\bf Acknowledgments.}
I would like to thank Yohsuke Matsuzawa for constant conversations to complete the original version of this article.
 I would  like to thank Fei Hu, Sheng Meng and Long Wang for useful comments, and Keiji Oguiso for answering questions, and Meng Chen and De-Qi Zhang for their encouragement.
 Finally, I am thankful to the referee for a careful reading of the article and the many suggestions.
\section{Preliminaries}\label{Preliminaries}
\begin{notation}
Let $X$ be a normal projective variety
Let $\NS_\bR(X):=\NS(X)\otimes_\bZ \bR$, where $\NS(X):=\Pic(X)/\Pic^0(X)$ is the usual N\'eron-Severi group of $X$.

Let $f: X\to X$  be a surjective endomorphism of $X$.
Denote by
$$
         \Alb(X):=\Pic^0(\Pic^0(X)_\mathrm{red})
$$
which is an abelian variety.
Then there is an albanese morphism $\alb_X: X\to \Alb(X)$ such that: $\alb_X(X)$ generates $\Alb(X)$ and for every morphism $\varphi: X \to A$ from $A$ to an abelian variety $A$, there exists a unique morphism $\psi: \Alb(X) \to A$ such that $\varphi=\psi\circ\alb_X$ (cf. \cite[Remark 9.5.25]{Fantechi05}).
By the above universal property, $f$ descends to surjective endomorphisms on $\Alb(X)$.
The irregularity $q(X)$ of $X$ is defined as:
$$ 
 q(X):= \dim \Alb(X).
$$

A surjective endomorphism $f: X\to X$  is said to be {\it int-amplified} if $f^*H-H=L$ for some ample Cartier divisors $H$ and $L$, or equivalently, if all eigenvalues of $f^*|_{\NS_{\bR}(X)}$ are of modulus greater than 1 (cf. \cite[Theorem 1.1]{Meng20}).

The Iitaka dimension   $\kappa(X,D)$ of a $\bR$-Cartier divisor $D$ on $X$ (cf. \cite{Nakayama04}) is the largest integer $k$ such that 
$$
             \limsup_{m\to\infty} \frac{h^0(X,\lfloor ml(D)D \rfloor )}{m^k}>0.
$$
otherwise, $\kappa(X,D)=-\infty$.
Here, $l(D)\in\{k\in\bZ_{>0}, |kD|\ne\emptyset\}$.
The Kodaira dimension of $X$ is $\kappa(X)=\kappa(X,K_X)$, where $K_X$ is the canonical divisor of $X$.
\end{notation}
The following lemma is well-known (cf. \cite[Lemma 5.6]{Matsuzawa20a}, \cite[Lemma 2.5]{MZ19}).
\begin{lemma}
\label{generically finite}
Let $\pi: X\dashrightarrow Y$ be a dominant rational map of projective varieties.
Let $f:X\rightarrow  X$ and $g:Y\rightarrow  Y$ be  surjective endomorphisms such that $g\circ\pi=\pi\circ f$.
Then the following hold.
\begin{itemize}
	\item[(1)] Suppose that $\pi$ is generically finite.
Then KSC holds for $f$ if and only if KSC holds for $g$.
	\item[(2)] Suppose $d_1(f)=d_1(g)$ and KSC holds for $g$.
Then KSC holds for $f$.
\end{itemize}
\end{lemma}
A normal projective variety $X$ is said to be {\it $Q$-abelian} if there is a finite surjective morphism $\pi: A\to X$ \'etale in codimension 1 with $A$ being an abelian variety.
\begin{theorem}\label{Q-abelian}
\begin{itemize}
	\item[(1)] KSC holds for any surjective endomorphism of an abelian variety.
	\item[(2)] KSC holds for any surjective endomorphism of a $Q$-abelian variety.\end{itemize}
\end{theorem}
\begin{proof}
(1) follows from \cite[Theorem 2]{Silverman17}, and (2) follows from \cite[Theorem 2.8]{MZ19}.
\end{proof}
The following theorem is due to Matsuzawa, Meng, Shibata, Zhang and Zhong \cite{MMSZZ22}.
\begin{theorem}\label{eigenvector}
\cite[Theorem 1.9(1)]{MMSZZ22}
Let $f: X\to X$ be a surjective endomorphism of a normal projective variety over a field $k$ of arbitrary characteristic.
Assume $d_1(f)>1$.
Then $f^*D\sim_\bR d_1(f)D$ for some nef $\bR$-Cartier divisor $D$.	
\end{theorem}
\begin{definition}
We say such $D$ in Theorem \ref{eigenvector} is a \emph{leading real eigendivisor} of $(X,f)$.
Note that KSC holds for $(X,f)$ if such $D$ has $\kappa(X,D)>0$ as follows.
\end{definition}
\begin{proposition}
\label{lead-eigen}
Let $X$ be a $\bQ$-factorial normal projective variety and $f \colon X \longrightarrow X$ a surjective morphism with $d_1(f)>1$.
\begin{enumerate}
\item If the lead real eigendivisor $D$ has $\kappa(X,D)>0$, then KSC holds for $f$.
\item Assume that $f\in \Aut(X)$.
Let $\bR$-divisors $D_{+}$  and $D_{-}$ respectively be the lead real eigendivisor of  $(X,f)$  and $(X,f^{-1})$ respectively.
If  $\kappa(X,D_{+}+D_{-})>0$, then KSC holds for $f$.
\end{enumerate}
\end{proposition}
\begin{proof}
It follows from Theorem \ref{eigenvector} and \cite[Proposition 5.2]{LM21}.
\end{proof}
\begin{definition}
Let $f$ be an automorphism of  a projective variety $X$ with $\dim X\ge3$ and let   $D_+$ and  $D_-$ be some $\bR$-Cartier divisors on $X$.
We say $(X,D_+,D_-)$ is {\it positive on  $(X,f)$} if the following hold:
 $$f^{*}D_{+} \sim_{\bR} d_1(f)D_{+}, (f^{-1})^{*}D_{-} \sim_{\bR} d_1(f^{-1})D_{-}, \kappa(X,D_{+}+D_{-})>0.$$
\end{definition}
\begin{remark}
Lesieutre and Satriano established in \cite{LS21} that KSC is true for hyper-K\"ahler varieties by using the pair $(X,D_+,D_-)$ which is positive on $(X,f)$ when $\kappa(X,D_++D_-)=\dim X$.
So an interesting question is asked as follows.
\end{remark}
\begin{question}
Does there exist an example of a projective variety $X$  with the pair  $(X,D_+,D_-)$ which is positive on  $(X,f)$ such that $\kappa(X,D_++D_-)<\dim X$?
\end{question}
\begin{theorem}
(cf.  \cite[Theorem 1.8]{LS21})
Let $f$ be an automorphism of a normal non-uniruled projective threefold $X$.
If the second Chern class $c_2(X)$ is strictly positive on $\Nef(X)$ and $q(X)=0$, then $f$ has finite order.
\end{theorem}
\begin{proof}
By \cite[Lemma 7.1]{BGRS17} we know that $\{D\in\Nef(X)~ |~ c_2(X)\cdot D\le m\}$ is compact for all $m\ge0$.
So the function $D\mapsto c_2(X)\cdot D$ achieves a minimum positive value on $\N^1(X)\cap\Amp(X)$ and this value is achieved by only finitely many $D_i$.
Taking the sum of these finitely many $D_i$, we obtain an ample class $A$  that is fixed by $f^*$.
Then by a Fujiki-Liberman type theorem (cf. \cite{Fujiki78,Lieberman78} or \cite[Theorem 1.4]{Li20}), some iterate $f^n$ lies in the connected component of the identity $\Aut^0(X)\subseteq\Aut(X)$.
For a smooth model $X'$ of $X$, the birational automorphism group $\Bir(X')$ contains $\Aut^0(X)$ as a subgroup.
By \cite[Theorem 2.1]{Han88}, $\Bir(X')$ is a disjoint union of abelian varieties of dimension equal to $q(X')=q(X)=0$.
We conclude that $f$ has finite order.
\end{proof}
\section{Invariant fibrations on projective varieties}\label{invariant}
Invariant fibrations play an important role in the study of rational maps in higher dimension, and the product formula of Dinh-Nguy{\^e}n-Truong \cite{DNT12} is useful in dealing with their dynamical degrees.
Let $\pi: X\dashrightarrow Y$ be a dominant rational map of projective varieties and $f: X\dashrightarrow X$ a dominant rational self-map.
Fix an ample divisor $H$ on $X$ and an ample divisor $H'$ on $Y$.
Now we give the following definition.
\begin{definition}
\label{relative dg}
{\it The first dynamical degree of $f$ relative to $\pi$} is defined by
\begin{equation*}
            d_1(f|_\pi)=\lim_{n\to\infty}((f^n)^*H\cdot\pi^*((H')^{\dim Y})\cdot H^{\dim X-\dim Y-1})^{1/n}.
\end{equation*}
\end{definition}
The following is due to the product formula (cf. \cite[Theorem 1.1]{DN11}).
 \begin{proposition}
(cf. \cite[Theorem 3.4]{LS21})
\label{x=y-1}
Let $\pi: X\dasharrow Y$ be a dominant rational map of normal projective varieties.
Suppose $f$ (resp. $g$) is a surjective endomorphism of $X$ (or $Y$) such that $g\circ\pi=\pi\circ f$.
If $d_1(f|_\pi)\le d_1(g)$ and KSC holds for $g$, then KSC also holds for $f$.
The condition $d_1(f|_\pi)\le d_1(g)$ holds in particular if $f$ is an automorphism and $\dim Y=\dim X-1$.
\end{proposition}
\begin{proof}
By the product formula \cite{Dang20,DN11,Truong20}, we have 
$$
    d_1(f)=\max\{d_1(g), d_1(f|_\pi)\}.
$$
Then $d_1(f)=d_1(g)$ if $d_1(f|_\pi)\le d_1(g)$.
Therefore, if $d_1(f|_\pi)\le d_1(g)$ and KSC holds for $g$, then KSC also holds for $f$ by Lemma \ref{generically finite}.

Now assume that $f$ is an automorphism and $\dim Y=\dim X-1$.
Another application of the product formula yields that $d_{\dim X}(f)=d_{\dim X-1}(g)d_1(f|_\pi)$.
Since $f$ is an automorphism, $d_{\dim X}(f)=1$, and so both terms on the right must be 1 as well.
So $d_1(f|_\pi)\le d_1(g)$.
\end{proof}
\begin{proof}[Proof of Theorem \ref{q(X)>X-1}]
 By \cite[Proposition 5.1]{CLO22}, we may assume that the Albanese morphism $\pi: X\to A$ is surjective and $\dim A=q(X)$.
 Notice that $f$ descents to a surjective endomorphism $f_A$ of $A$ by the universal property of the Albanese morphism.
If $\dim A=\dim X$, then KSC  holds for $(X,f)$ by Lemma \ref{generically finite} and Theorem \ref{Q-abelian}.
If $\dim A=\dim X-1$, then  KSC holds for $(X,f)$ by Theorem \ref{Q-abelian} and Proposition \ref{x=y-1}.
\end{proof}
The following is motivated by \cite[Proof of Proposition 1.7]{LS21}.
\begin{theorem}
\label{KSC-P1}
Let $f$ be an automorphism on a normal  projective variety $X$ with non-negative Kodaira dimension.
If $X$ admits a $f$-equivariant non-constant morphism $\pi: X\to \bP^1$, then $X$ does no have any dense  orbit.
\end{theorem}
\begin{proof}
By \cite[Theorem 13.2]{BM97}, there exists a resolution $\varphi: \widetilde{X}\to X$ and an automorphism $\widetilde{f}$ of $\widetilde{X}$ such that $\widetilde{f}\circ\varphi=\varphi\circ f$.
Then we may assume that $X$ is smooth.
Now assume that $f$ descends to an automorphism $g$ of $\bP^1$ as $\dim \bP^1=1$.
Let $Z\subset\bP^1$ be the locus of points $t$, where the fiber $X_t$ is singular.
Then $g(Z)=Z$.
Since $Z$ is a finite set, after replacing $f$ by a further iteration, we can assume $g$ fixes $Z$ point-wise.
By \cite[Theorem 0.2]{VZ01}, we know that $Z$ contains at least three points.
It follows that $g$ is the identity since it fixes at least three points of $\bP^1$.
In other words, there exists a rational function on $X$ that is invariant under some  iteration of $f$, which contradicts the fact that $X$ has a point with a dense orbit.
\end{proof}
\begin{remark}
\begin{itemize}
\item[(1)] 	If $f$ is a surjective endomorphism of a smooth projective 	variety $X$ with $\kappa(X)\ge0$ admits a surjective morphism $\pi: X\to Y\cong \bP^1$,  and  let $g: Y\to Y$ be an automorphism such that $\pi\circ f=g\circ\pi$, then $g$ is of finite order by \cite[Proposition 2.3]{NZ09}.
\item[(2)] In general, we would like to ask whether that $X$ does not have any dense orbit if $f$  is a birational self-map of  a normal  projective variety $X$ with $\kappa(X)\ge0$ admits a $f$-equivariant map $\pi: X\dasharrow \bP^1$.
\end{itemize}
\end{remark}
\begin{proposition}
\label{threefold}
Let $\pi: X\dashrightarrow Y$ be a dominant rational map of normal projective varieties with $3=\dim X>\dim Y\ge1$.
Let $f: X\rightarrow X$ and $g: Y\rightarrow Y$ be  surjective endomorphisms such that $g\circ\pi=\pi\circ f$.
Suppose $f$ is an automorphism.
Then to show KSC  for $(X,f)$, we only need to assume that $Y$ is $\bP^1$ or an elliptic curve.	
In particular, if $\kappa(X)\ge0$, then we may KSC to the case that  $Y$ is an elliptic curve.
\end{proposition}
\begin{proof}
If  $\dim Y=2$, then Conjecture \ref{KSC} holds for $(Y,g)$ by \cite[Theorem 1.3]{MZ19}.
Then Conjecture \ref{KSC} holds for $(X,f)$ by Proposition \ref{x=y-1}.
Then we may assume that $\dim Y=1$.
If $g(Y)\ge2$ ( i.e., $\kappa(Y)>0$), then $Y$ does no have any dense $g$-orbit by Remark \ref{non-dense}.
So Conjecture \ref{KSC} is true for $(X,f)$.
Therefore, we may assume that $Y$ is $\bP^1$ or an elliptic curve.
If $\kappa(X)\ge0$ and $\pi$ is a morphism, then the proof follows from Theorem \ref{KSC-P1}.
\end{proof}
\section{Automorphisms on projective threefolds}
\label{MainSect}
Now we  quote the definition of a weak Calabi-Yau variety in \cite[Definition 2.4]{HL21}.
\begin{definition}\label{defn-wcy}
A normal projective variety $X$ is called a \emph{weak Calabi-Yau variety}, if
\begin{itemize}
\item $X$ has only canonical singularities;
\item the canonical divisor $K_X\sim 0$; and
\item the \emph{augemented irregularity} $\widetilde{q}(X)=0$.
\end{itemize}
Here, the augumented irregularity $\widetilde{q}(X)$ of $X$ is defined as follows:
$$\widetilde{q}(X)=\sup\big\{ q(Y) : Y\to X \text{ is  a finite surjective and \'etale in codimension one}\big\}.$$
\end{definition}
The following {\it special MRC fibration}  is due to Nakayama \cite{Nakayama10}.
\begin{definition}
\cite[Definition 2.10]{HL21}
Given	a projective variety $X$, a dominant rational map $\pi: X\dashrightarrow Z$ is called the  special MRC fibration of $X$, if it satisfies the following conditions:
\begin{itemize}
\item[(1)] The graph $\Gamma_\pi\subseteq X\times Z$ of $\pi$ is equidimensional over $Z$.
\item[(2)] The general fibers of $\Gamma_\pi\to Z$ are rationally connected.
\item[(3)] $Z$ is a non-uniruled normal projective variety.
\item[(4)] If $\pi':X\dasharrow Z'$ is a dominant rational map satisfying (1)-(3), then there is a birational morphism $v: Z'\to Z$ such that $\pi=v\circ\pi'$.
\end{itemize}
\end{definition}
\begin{proposition}
\cite[Proposition 4.14]{Nakayama10}
\label{Chow-red}
Let $\pi:X\dasharrow Y$ be a dominant rational map from a projective variety $X$ to a normal projective variety $Y$.
Then there exists a normal projective variety $T$ and a birational map $\mu: Y\dasharrow T$ satisfying the following conditions:
\begin{itemize}
\item[(1)] The graph $\gamma_{\mu\circ\pi}:\Gamma_{\mu\circ\pi}\to T$ of $\mu\circ \pi$ is equi-dimensional.
\item[(2)] Let $\mu': Y\dasharrow T'$ be a birational map to another normal projective variety $T'$ such that the graph $\gamma_{\mu'\circ\pi}:\Gamma_{\mu'\circ\pi}\to T'$ of $\mu'\circ\pi$ is equi-dimensional.
Then there exists a birational morphism $\nu: T'\to T$ such that $\mu=\nu\circ \mu'$.
\end{itemize}
\end{proposition}
We call the composition $\mu\circ\pi: X\dasharrow T$ above satisfying Proposition \ref{Chow-red} (1)-(2) the \emph{Chow reduction} of $\pi: X\dasharrow Y$, which is unique up to isomorphism.
\begin{theorem}\label{descent}
Let $\pi: X\dasharrow Z$ be the special MRC fibration, and $f$ be an automorphism of $X$.
Then there exists a  birational morphism $p: W\to X$ and an automorphism $f_W\in\Aut(W)$ and an equi-dimensional surjective morphism $q: W\to Z$ satisfying the following conditions:
\begin{itemize}
\item[(1)]  $W$ is a normal projective variety.
\item[(2)] A general fiber of $q$ is rationally connected.
\item[(3)] $W$ admits $f_W$-equivariant fibration over $X$ and $Z$.	
\end{itemize}
\end{theorem}
\begin{proof}
It follows from \cite[Theorem 4.19]{Nakayama10} or \cite[Lemma 2.11]{HL21}.	
\end{proof}
\begin{proof}[Proof of Theorem \ref{reduced-auto}]
We first suppose $K_X\sim_\bQ0$.
Then by \cite[Lemma 2.7]{HL21}, there exists a finite surjective morphism   $\pi: \widetilde{X} \to X$ and an automorphism $\widetilde{f}$ of $\widetilde{X}$ such that the following statements hold.
\begin{itemize}
	\item $X\cong Z\times A$ for a weak Calabi-Yau variety $Z$ and an abelian variety $A$.
	\item $\dim A=\widetilde{q}(X)$.
	\item There are automorphisms $\widetilde{f}_Z$ and $\widetilde{f}_A$ of $Z$ and $A$ respectively, such that the following diagram commutes:
$$
\xymatrix{
X  \ar[d]_f  &    \widetilde{X} \ar[l]_{\pi}  \ar[r]^{\cong}  \ar[d]_{\widetilde{f}} & Z\times A \ar[d]^{\widetilde{f}_Z\times \widetilde{f}_A}\\
X   &  \widetilde{X} \ar[r]^{\cong}  \ar[l]_{\pi} & Z\times A.\\
}
$$
\end{itemize}
If $\widetilde{q}(X)=3$, then it follows from Theorem \ref{Q-abelian} and Proposition \ref{generically finite}.
If $\dim Z>0$, then $\dim Z\ge2$ since  $\widetilde{q}(Z)=0$ and $K_Z\sim0$.
If $\dim Z=2$, then $\dim A=1$.
Consider the natural projection $\mathrm{pr}_1: \widetilde{X}\to Z$.
Then by Propositions  \ref{generically finite} and \ref{threefold}, KSC  is true for $(X,f)$.
Now assume that $\dim Z=3$ and $f$  is imprimitive.
 Then there is a rational map $\pi: X\dasharrow Y$ and a birational map $g: Y\dasharrow Y$ such that $\pi\circ f=g\circ \pi$.
By Proposition \ref{Chow-red}, there exists a birational morphism $\mu: Y\to Z$ such that $\pi'=\mu\circ \pi: X\dasharrow Z$ is the Chow reduction of $\pi$.
Then $f$ descents to  an automorphism $h$ of $Z$ by Proposition \ref{Chow-red}.
 By taking the graph of $\pi'$, it suffices to consider the case when $\pi'$ is morphism by  Lemma  \ref{generically finite}.
By Proposition \ref{threefold}, we may assume that  $Y$ is  an elliptic curve.
This completes the proof of Theorem \ref{reduced-auto}(1) as $X$ has trivial Albanese.

Now we assume that $\kappa(X)=-\infty$.
Consider the special MRC fibration $\pi: X\dasharrow Z$.
By Theorem \ref{descent} and Lemma  \ref{generically finite}, we may assume that $\pi: X\to Z$ is $f$-equivariant morphism.
There are an equi-dimensional surjective morphism $q: W\to Z$ and the general fiber of $q$ is rationally connected.
If $\dim Z=0$, then $X$ is rationally connected.
Now assume that $\dim Z>0$.
Here $Z$ is non-uniruled.
By Proposition \ref{threefold}, $Z$ is an elliptic curve.
So $q(X)>0$.
Take the $f$-equivariant resolution for $(X,f)$,  then by Proposition \ref{generically finite}, we assume that $f$ is an automorphism of a  smooth projective threefold with $q(X)>0$.
Therefore, KSC is true for $f$ by \cite[Theorem 1.6(2)]{CLO22}.
\end{proof}
\section{On projective varieties with Picard number two}
\label{Pic=2}
An interesting example of  projective varieties with Picard number two is a projective bundle $X$ over a normal projective variety $Y$ with Picard number $\rho(Y)=1$.
Then we ask the following question.
\begin{question}\label{que-proj-bundle}
Is KSC true for any surjective endomorphism of a projective bundle $X$ over a normal projective variety $Y$ with $\rho(Y)=1$?
\end{question}
\begin{remark}
Question \ref{que-proj-bundle} is affirmative  if $Y$ is a smooth Fano variety (cf. \cite{LM21,LS21}).
\end{remark}
To study KSC on projective varieties with Picard number two, we  first give the following result, which is motivated by Sano's theorem \cite{Sano17} on arithmetic and dynamical degrees. 
\begin{proposition}
\label{reduce-Pic2}
Let $f$ be a surjective endomorphism of a normal projective variety $X$ with the Picard number $\rho(X)=2$.
Let $\lambda_1$ and $\lambda_2$ be two eigenvalues of $f^*|_{\NS_{\bR}(X)}$ such that $|\lambda_1|\ge|\lambda_2|$.
To prove KSC is true for $(X,f)$, then we may assume that $f$ is int-amplified (i.e. $d_1(f)=\lambda_1>\lambda_2>1)$ or $\lambda_2=1$.
Moreover, we may assume that $f^*|_{\NS_{\bR}(X)}$ is diagonalizable with positive eigenvalues $p$ and $q$ with $p>q\ge1$.
\end{proposition}
\begin{proof}
It is known that $1\le \alpha_f(x)\le d_1(f)$ by \cite[Theorem 4]{KS16} and \cite[Theorem 1.4]{Matsuzawa20b}.
Now we may assume that $d_1(f)>1$.
Note that $|\lambda_1|=d_1(f)$.
If $|\lambda_1|=|\lambda_2|>1$ or $|\lambda_1|>1>|\lambda_2|$, then the arithmetic degree $\alpha_f(x)=d_1(f)$ by \cite[Theorem 1.1]{Sano17}.
Then we may assume that $|\lambda_1|>|\lambda_2|\ge1$.
After replacing $f$ by $f^2$, $f^*|_{\NS_{\bR}(X)}$ is diagonalizable with positive eigenvalues $d_1(f)$ and $\lambda_2$ with $d_1(f)>\lambda_2\ge1$ as $\rho(X)=2$.
If $\lambda_2>1$, then $f$ is int-amplified by  \cite[Theorem 1.1]{Meng20}.
This completes the proof of  Proposition \ref{reduce-Pic2}.
\end{proof}
\begin{corollary}
\label{CY-Pic2}
Let $f$ be a surjective endomorphism of a weak Calabi-Yau variety $X$ with  $\rho(X)=2$.
Let $\lambda_1$ and $\lambda_2$ be two eigenvalues of $f^*|_{\NS_{\bR}(X)}$ such that $|\lambda_1|\ge|\lambda_2|$.
To prove KSC is true for $(X,f)$, then we may assume that $\lambda_2=1$.
Moreover, we may assume that $f^*|_{\NS_{\bR}(X)}$ is diagonalizable with positive eigenvalues $p>q\ge1$.
\end{corollary}
\begin{proof}
Since $K_X$ is pseduoeffective,  the proof follows from  Proposition \ref{reduce-Pic2},  \cite[Theorem 1.9]{Meng20} and Theorem \ref{Q-abelian}.
\end{proof}
The following is due to \cite[Lemma 10.4]{MZ19}.
\begin{proposition}
\label{MZ-10.4}
Let $f: X\to X$ be a surjective endomorphism of a normal projective variety $X$.
Suppose $f^*|_{\N^1(X)}$ is diagonalizable with positive eigenvalues $p\ge q\ge1$, an no other eigenvalues.
Let $H$ be an ample Caritier divisor.
Then $H=A+B$ for some nef $\bR$-Cartier divisors $A$ and $B$ such that $f^*A\equiv pA$ and $f^*B\equiv qB$.
\end{proposition}
\begin{proof}
If $p=q$, then $f^*|_{\N^1(X)}=p\id$ and we may take $A=H$ and $B=0$.
Assume $p>q$.
Let $\varphi:=f^*|_{\N^1(X)}$.
Let $A=\lim_{i\to+\infty}\varphi^i(H)/p^i$ and $B=\lim_{i\to\infty}q^i\varphi^{-i}(H)$.
Since $\varphi$ is diagonalizable with only eigenvalues $p$ and $q$, the above limits are $\bR$-Cartier and $H=A+B$.
It is clear that $\varphi(A)=pA$ and $\varphi(B)=qB$.
Note that $A$ and $B$ are limits of ample divisors.
So $A$ and $B$ are nef.
\end{proof}
\begin{remark}
To prove KSC is true for projective varieties with Picard number two, we wish construct a canonical height function $\hat{h}_D$ associated with some divisors $D$ ( e.g. the lead real eigendivisor of $(X,f)$)
\begin{equation*}
	 \hat{h}_D(x)=\lim_{n\to\infty}\frac{h_D(f^n(x))}{d_1(f)^{\dim X}} 
\end{equation*}
which  is positive at every point $p\in X(\overline{\bQ})$ with a Zariski dense $f$-orbit by using $A$ and $B$ in Proposition \ref{MZ-10.4}.
\end{remark}
Now we quote KSC for the TIR case on normal projective varieties admitting an int-amplified endomorphism  in \cite{MMSZ23} as follows.
\begin{theorem}
\label{TIR}
Let $f$ be a surjective endomorphism of a normal projective variety $X$ admitting an int-amplified endomorphism.
To show KSC holds for $f$,it suffices to show KSC is true for that $(X,f)$ is the TIR case as follows:
\begin{itemize}
\item[(1)] $\kappa(X,-K_X)=0$;
\item[(2)] $f^*D_1=d_1(f)D_1$ for some reduced effective Weil divisor $D_1\sim_{\bQ}-K_X$;
\item[(3)] There is an $f$-equivariant Fano contraction $\pi: X\to Y$ with $d_1(f)>d_1(f_Y)(\ge1)$, where  $f_Y$ is induced by $f$ on $Y$.
\end{itemize}	
\end{theorem}
 \begin{proof}
It follows from \cite[Theorem 1.7]{MZ19} and \cite[Lemma 6.4]{MMSZ23}.
\end{proof}
As a start to addressing Question \ref{que-proj-bundle} when $\dim Y>1$,  we show the following result.
\begin{theorem}
\label{projective-bundle}
Let $f$ be a surjective endomorphism of a projective bundle $X$ over a normal projective variety $Y$ with $\rho(Y)=1, \pi: X\to Y$ the projection.
Replacing $f$ with its iterate, $f$ descents to  $g$ on $Y$.
To prove KSC is true for $(X,f)$,we may assume that one of the following case holds:
\begin{itemize}
	\item[(1)]  $d_1(f)=d_1(f|_\pi)>d_1(g)>1$, and so $f$ is int-amplified.
	\item[(2)]  $d_1(f)=d_1(f|_\pi)>d_1(g)=1$,  and the morphisms between fibers of $\pi$ induced by $f$ are not isomorphism.
\end{itemize}
Notice that the degree of $f$ is greater than one.
\end{theorem}
\begin{proof}
Replacing $f$ with its iterate, by \cite[Lemma 2.3]{LM21} we may assume that $f$ induces an endomorphism $g: Y\to Y$ such that $g\circ \pi=\pi\circ f$  (cf. discussion before \cite[Theorem 2]{Amerik01}).
Since $\rho(Y)=1$, $g$ is polarized and KSC is true for $g$ by  \cite[Theorem 5]{KS16}.
Therefore, we may assume $d_1(f)=d_1(f|_\pi)>d_1(g)$ by Lemma \ref{generically finite} and the product formula on dynamical degrees.
Note that $\NS_{\bR}(X)=\bR\mathcal O_X(1)\oplus \pi^*\NS_{\bR}(Y)$.
If $d_1(g)>1$, then the eigenvalues of $f^*:\NS_{\bR}(X)\to \NS_{\bR}(X)$ are $d_1(f)$ and $d_1(g)$ and, then they have modulus larger than one.
Thus $f$ is  int-amplified  by \cite[Theorem 1.1]{Meng20}. 
As a result, $\deg f>1$ by \cite[Lemma 3.7]{Meng20}.
Now Suppose $d_1(f)>d_1(g)=1$. 
Then the morphisms between fibers of $\pi$ induced by $f$ are not isomorphism.
Indeed, let $\pi^{-1}(y)$ be a closed fiber.
Since $\NS_{\bR}(X)=\bR\mathcal O_X(1)\oplus \pi^*\NS_{\bR}(Y)\cong\bR^2$ and $f^*$ fixes $\pi^*\NS_{\bR}(Y)$, we see that $f^*\mathcal O_X(1)=d_1(f)\mathcal O_X(1)+\pi^*D$ for some divisor $D$ on $Y$.
\begin{equation*}\begin{split}
  (f_*[\pi^{-1}(y)]\cdot\mathcal O_X(1)^{\dim X-\dim Y})&=([\pi^{-1}(y)]\cdot f^*\mathcal O_X(1)^{\dim X-\dim Y})
  \\&=([\pi^{1}(y)]\cdot (d_1(f)\mathcal O_X(1))^{\dim X-\dim Y})
  \\&=d_1(f)^{\dim X-\dim Y}.
\end{split}\end{equation*}
This shows that the degree of the morphism $f: \pi^{-1}(y)\to \pi^{-1}(g(y))$ is $d_1(f)^{\dim X-\dim Y}$ which is greater than one. 
So we also have $\deg f>1$.
\end{proof}
\begin{remark}
The two reduced cases in Theorem \ref{projective-bundle} are very difficult.
\begin{itemize}
\item[(1)] In the first case, it suffices to show  KSC for the TIR case by  Theorem \ref{TIR}.
\item[(2)] In the second case we  need to study the relations of arithmetic degrees and relatively dynamical degrees on projective bundles, but the morphisms on fibers induced by $f$ may not be an endomorphism.
In general, let $\pi: X\to Y$ be a surjective morphism and a surjective endomorphism  $f$ of $X$ descending to a surjective endomorphism  $g$ on $Y$.
Then a question is asked as follows.
\end{itemize}
\end{remark}
\begin{question}\label{que-f>g}
Let  $\pi: X\to Y$ be a surjective morphism between normal projective varieties $X$ and $Y$ with $\dim X>\dim Y>0$ and a surjective endomorphism  $f$ of $X$ descents to $g$ on $Y$.
Is KSC true for $(X,f)$ when $d_1(f)>d_1(g)$?
\end{question}
\begin{remark}
When $\dim Y=1$, Question \ref{que-f>g} is affirmative  if $\kappa(X)\ge0$, $Y=\bP^1$ and $f$ is an automorphism (cf. Theorem \ref{KSC-P1}).
\end{remark}

\end{document}